\documentclass[12pt,a4paper]{article}
\usepackage{a4wide,amsmath,amsthm,amssymb,amsfonts,color}
\usepackage{mathtools}
\usepackage[backref=page]{hyperref}
\hypersetup{
 colorlinks,
 citecolor=green,
 linkcolor=blue,
 urlcolor=Blue}
\mathtoolsset{showonlyrefs}
\usepackage[all]{xy}
\usepackage[utf8]{inputenc}
\usepackage{enumitem,MnSymbol,authblk}

\usepackage{mathrsfs}
\usepackage{cancel}

\newtheorem{thm}{Theorem}[section]
\newtheorem{prop}[thm]{Proposition}
\newtheorem{lma}[thm]{Lemma}
\newtheorem{cor}[thm]{Corollary}
\newtheorem{dfn}[thm]{Definition}

\theoremstyle{remark}
\newtheorem{rmkk}[thm]{Remark}
\newtheorem{exe}[thm]{Example}
\newenvironment{rmk}{\begin{rmkk}\rm}{\qee\end{rmkk}}

\newcommand{\qee}{\mbox{\hspace{0.2mm}}\hfill$\triangle$}

\newcommand{\C}{\mathbb{C}}
\newcommand{\Q}{\mathbb{Q}}
\newcommand{\Pj}{\mathbb{P}}

\newcommand{\N}{\mathbb{N}}
\newcommand{\cO}{\mathcal{O}}

\newcommand{\codim}{\operatorname{codim}}

\begin{document}
\begin{center}
{\bf \large     On the Hodge conjecture for  quasi-smooth \\[5pt] intersections  in toric  varieties} 

\bigskip
{\sc  Ugo Bruzzo$^\ddag$  and William D. Montoya$^\P$}

\medskip
{\small
$^\ddag$ SISSA (Scuola Internazionale Superiore di Studi Avanzati),  \\ Via Bonomea 265, 34136 Trieste, Italy;   \\[2pt]
Departamento de Matem\'atica, Universidad Federal da Para\'iba, \\ Campus I, Jardim Universit\'ario, 58051-900, Jo\~ao Pessoa, PB, Brazil; \\[2pt]
INFN (Istituto Nazionale di Fisica Nucleare), Sezione di Trieste;\\[2pt] 
IGAP (Institute for Geometry and Physics), Trieste; \\[2pt]
Arnold-Regge Center for Algebra, Geometry \\ and Theoretical Physics, Torino; \\[2pt]
$^\P$ Instituto de Matem\'{a}tica, Estat\'{i}stica e Computa{\c c}\~{a}o Cient\'{i}fica, \\ Universidade Estadual de Campinas,
Rua S\'ergio Buarque de Holanda   651, \\ 
13083-859, Campinas, SP, Brazil
}
\end{center}

\bigskip 

\begin{abstract} 
We establish the Hodge conjecture for some subvarieties of a class of toric varieties. First we study quasi-smooth intersections in a projective simplicial toric variety, which is a suitable notion to generalize smooth complete intersection subvarieties in the toric environment, and in particular quasi-smooth hypersurfaces. We show that under appropriate conditions, the Hodge Conjecture holds for  a very general quasi-smooth intersection subvariety, generalizing the work on quasi-smooth hypersurfaces of the first author and Grassi in \cite{bg2}. We also show that the Hodge Conjecture  holds asymptotically for suitable quasi-smooth hypersurface in the Noether-Lefschetz locus, where ``asymptotically'' means
that the degree of the hypersurface is big enough. This extendes to toric varieties  Otwinowska's result  in \cite{Annia}.\end{abstract}
    
 \vfill
 
 \noindent  \parbox{.6\textwidth}{\hrulefill} \\
{\footnotesize Date: Revised \today  \\[-3pt]
Email: {\tt  bruzzo@sissa.it, montoya@unicamp.br }
}

\newpage

\section{Introduction}
A projective simplicial toric variety $\Pj^d_{\Sigma}$ satisfies the Hodge Conjecture, i.e., every cohomology class in $H^{p,p}(\Pj^{d}_{\Sigma},\Q)$ is a linear combination of algebraic cycles. On the one hand, by the Lefschetz hyperplane theorem, the Hodge conjecture holds true for every hypersurface and $p<\frac{d-1}{2}$ and by the hyperplane Lefschetz theorem, and by Poincar\'e duality, also for $p>\frac{d-1}{2}$. Moreover, by Theorem 1.1 in \cite{bg2}, when $p=\frac{d-1}{2}$, $d=2k+1$ and $\Pj_{\Sigma}^{2k+1}$ is an Oda variety with an ample class $\beta$ such that  $k\beta-\beta_0$ is nef, where $\beta_0$ is the anticanonical class, the Hodge conjecture with rational coefficients holds for a very general hypersurface in the linear system $|\beta|$. 

The notion of Oda varieties was introduced in \cite{bg-nl}. Let us recall that the Cox ring of a toric variety $\Pj_\Sigma$ is graded
over the class group $\operatorname{Cl}(\Pj_\Sigma)$, and that one has an injection $\operatorname{Pic}(\Pj_\Sigma) \to\operatorname{Cl}(\Pj_\Sigma)$.

\begin{dfn}\label{A1}  Let $\Pj_\Sigma$ be a toric variety with Cox ring $S$. $\Pj_\Sigma$ is said to be     an Oda variety if the multiplication morphism $S_{\alpha_1 }\otimes S_{\alpha_2} \to S_{\alpha_1+\alpha_2}$
is surjective whenever the classes $\alpha_1$ and $\alpha_2$ in $\operatorname{Pic}(\Pj_\Sigma)$ are
ample and nef, respectively.
\end{dfn}
 
In \cite{Mavlyutov} Mavlyutov proved a Lefschetz type theorem for quasi-smooth intersection subvarieties, and moreover  using  the ``Cayley trick''  he related the cohomology of a quasi-smooth subavariety $X=X_{f_1}\cap \dots \cap X_{f_s}\subset \Pj^{d}_{\Sigma}$ to the cohomology of a quasi-smooth hypersurface $Y\subset \Pj^{d+s-1}_{\Sigma}$.  This allows us to prove a Noether-Lefschetz type theorem, namely:

\noindent \textbf{Theorem \ref{NL}.}{\em Let $\Pj^d_{\Sigma}$ be an Oda projective simplicial toric variety. For a very general quasi-smooth intersection subvariety $X$ cut off by $f_1,\dots f_s$ such that $d+s=2(\ell +1)$ and $$\sum_{i=1}^s \deg (f_i)- \beta_0$$ is nef,  one has 
$$H^{\ell +1-s,\ell +1-s}(X,\Q)=i^*\left( H^{\ell +1-s,\ell +1-s}(\Pj^d_{\Sigma},\Q)\right).$$ }

From this one obtains the  following result about the Hodge conjecture for quasi-smooth intersections.

\noindent \textbf{Corollary \ref{hodge1}.} {\em  If  $\Pj_{\Sigma}^d$ is an Oda projective simplicial toric variety,
the Hodge Conjecture holds for a very general  quasi-smooth intersection subvariety $X$ cut off by $f_1,\dots f_s$ such that $d+s$ is even  and  $\sum_{i=1}^s \deg (f_i)- \beta_0$ is nef.}

Let $T$ be the open subset of $\vert\beta\vert$ corresponding to quasi-smooth hypersurfaces, and let $\mathcal{H}^{2k}=R^{2k}\pi_\ast\C\otimes_\C\cO_T$ be
the Hodge bundle on $T$; here $\pi\colon\mathcal X \to T$ is the tautological family on $T$, and $d=2k+1$. 
We restrict $\mathcal H^{2k}$ to a contractible open subset $U\subset T$. The bundle $\mathcal{H}^{2k}$
has a Hodge decomposition
$$\mathcal{H}^{2k} = \bigoplus_{p+q=2k}\mathcal H^{p,q}$$
but this is not holomorphic. On the other hand, the bundles that make up the Hodge filtration
$$F^p\mathcal{H}^{2k} = \bigoplus_{p=0}^{2k}\mathcal H^{2k-p,p}$$
are holomorphic; to see this one can use the {\em period map} (which in particular we write for $p=k$)
$$\mathcal P^{k,2k}\colon U \to \operatorname{Grass}(b_k,H^{2k}(X_{u_0},\C))$$
where $b_k = \dim F^kH^{2k}(X_{u_0},\C)$ for a fixed point $u_0\in U$; this map sends $f\in U$ to the subspace $F^kH^{2k}(X_f,\C) \subset H^{2k}(X_f,\C) = H^{2k}(X_{u_0},\C)$.
This map is holomorphic (see \cite{LiuZhuang} and \cite[Prop.~3.4]{bm1}). But, by the very definition of the period map (see also \cite{v1}, Section 10.2.1 for the smooth case)
$$ F^k\mathcal{H}^{2k} \simeq (\mathcal P^{k,2k})^\ast \mathcal U_k,$$
where $\mathcal U_k$ is the tautological bundle on the Grassmannian $\operatorname{Grass}(b_k,H^{2k}(X_{u_0},\C))$, so that the bundles $F^k\mathcal{H}^{2k} $
are indeed holomorphic.  

Pushing ahead the ideas developed in \cite{bm1} and  \cite{bm}, let  $\lambda_f$ be a nonzero class
in the primitive cohomology $H^{k,k}(X_f,\Q)/H^{k,k}(\Pj^{2k+1}_{\Sigma},\Q)$,  and let $U$ be a contractible open subset of $T$ around $f$, so that $\mathcal{H}^{2k}_{\vert U}$ is constant. Moreover, let $\lambda\in \mathcal{H}^{2k}(U)$ be the section defined by $\lambda_f$ and let $\bar{\lambda}$ be  its image in $(\mathcal{H}^{2k}/F^k\mathcal{H}^{2k})(U)$.  One has

\begin{prop} The local Noether-Lefschetz loci can be defined as \label{lnl} $$N^{k,\beta}_{\lambda,U}:=\{G\in U \mid \bar{\lambda}_{G}=0\}$$ 
where $\beta=\deg (f)$. \end{prop}

The following result is  Theorem 1.2 in  \cite{bm}.

\noindent\textbf{Theorem.}  {\em Let $\Pj^{2k+1}_\Sigma$ be an Oda variety with an ample class $\beta$
such that $k\beta-\beta_0=n\eta$, where $\beta_0$ is the anticanonical class, $\eta$ is a primitive ample class,
and   $n\in \N$.  Let
\begin{equation}\label{mbeta} m_\beta = \max \{i\in\N\,\,\vert\,i\eta \le \beta \}. \end{equation}
For every positive $\epsilon$ there is a positive $\delta$
such that for every $m\ge \max(\frac1\delta,m_\beta)$ and $d\in[1,m\delta]$,  and every nontrivial Hodge class $\lambda\in F^k\mathcal H^{2k}(U)$ such that
$$\codim N^{k,\beta}_{\lambda,U} \leq d\frac{m_{\beta}^k}{k!},$$     for every $f\in N_{\lambda,U}^{k,\beta}$,  there exists a $k$-dimensional variety $V\subset X_f$ with $\deg V\leq d$.
Here $\deg V$ and $\deg X_f$ are taken with respect to the ample divisor $\eta$, i.e.,
$$ \deg V = [V]\cdot\eta^k,\qquad \deg X_f = [X_f]\cdot \eta.$$ }

Based on this, in this paper we obtain the following result.

\noindent\textbf{Theorem \ref{hodge2}.}  {\em 
Under the same hypotheses of the previous theorem, 
if $V\subset X_f$ is a nonempty quasi-smooth   intersection subvariety of $\Pj^{2k+1}$ for some $f\in N_{\lambda,U}^{k,\beta}$,   then there exists $c\in \Q^*$ such that $\lambda_f=c\lambda_V$,
where $\lambda_V$ is the class of $V$ in $H_{\mbox{\footnotesize prim}}^{k,k}(X_f,\Q)$.}

In other words, $\lambda_f$ is algebraic.

In his paper \cite{Dan} A.~Dan proves a form of  our Theorem \ref{hodge2} for smooth hypersurfaces in   odd-dimensional projective spaces
$\mathbb P^{2k+1}$ which is not asymptotic. So our result is more general  in two ways, as we consider {\em quasi-smooth intersections} in {\em toric varieties;} however, our result is asymptotic.
 
 \medskip
\textbf{Acknowledgement.}
We thank Paolo Aluffi  for useful discussions, and Antonella Grassi for developing with the first author the foundations on which this work is based. We are very thankful to the referee
for her/his very careful reading, and the many suggestions and remarks which allowed us to greatly improve the presentation of this paper. The
second author acknowledges support from FAPESP postdoctoral grant No.~2019/23499-7.

\section{Very general quasi-smooth intersections}

Let $f_1,\dots,f_s$  be  weighted homogeneous polynomials in the Cox ring $S = \C[x_1,\dots,x_n]$ of $\Pj_{\Sigma}^d$.  Their zero locus $V(f_1,\dots,f_s)$  defines a closed subvariety $X\subset \Pj^{d}_{\Sigma}$. Let $U(\Sigma)= \mathbb A^n -Z(\Sigma)$, where $Z(\Sigma)$ is the irrelevant locus, i.e., $Z(\Sigma) = \operatorname{Spec}B$, where 
$B$ is the irrelevant ideal.

\begin{dfn} {\rm \cite{Mavlyutov}}  $X$ is a codimension $s$ quasi-smooth intersection if $V(f_1,\dots,f_s)\cap U(\Sigma)$ is either empty or a smooth   interesection subvariety of codimension $s$ in $U(\Sigma)$.
\end{dfn}
This notion generalizes that of smooth complete intersection in a projective space.  {For $s=1$ it reduces to the notion of {\em quasi-smoooth hypersurface,}
see Def.~3.1 in \cite{bc}. If we regard $\Pj_{\Sigma}^d$ as an orbifold, then a hypersurface $X$ is quasi-smooth when it is a sub-orbifold of $\Pj_{\Sigma}^d$;
heuristically, ``$X$ has only singularities coming from the ambient variety.''}

%
%

We also have a Lefschetz type theorem in this context.

\begin{prop}[\cite{Mavlyutov} Proposition 1.4] Let $X\subset \Pj_{\Sigma}^d$ be a closed subset, defined by homogeneous polynomials $f_1,\dots f_s\in B$. Then the natural map
$i^{*}: H^i(\Pj_{\Sigma}^d)\rightarrow H^i(X)$ is an isomorphism for $i<d-s$ and an injection for $i=d-s$.
In particular, this is true if  the hypersurfaces cut by the polynomials $f_i$ are ample.

\end{prop}

Hence if $p\neq \frac{d-s}{2}$ every cohomology class in $H^{p,p}(X)$ is a linear combination of algebraic cycles. So let us see what happens when $p=\frac{d-s}{2}$. The idea   is to relate the Hodge structure of a quasi-smooth intersection  variety $X=X_{f_1}\cap \dots \cap X_{f_s}$ in $\Pj^d_{\Sigma}$ with the Hodge structure of a quasi-smooth hypersurface $Y$ in a toric variety $\Pj^{d+s-1}_{X,\Sigma}$ whose fan depends on $X$ and $\Sigma$.

\begin{prop}\label{prop}  Let $X=X_1 \cap \dots\cap X_s$ be quasi-smooth intersection subvariety in $\Pj_{\Sigma}^d$ cut off by homogeneous polynomials $f_1\dots f_s$.   There exists a projective simplicial toric variety $\Pj^{d+s-1}_{X,\Sigma}$ and a quasi-smooth hypersurface $Y\subset \Pj^{d+s-1}_{X,\Sigma}$ such that for  $p\neq \frac{d+s-1}{2}, \frac{d+s-3}{2} $
$$H^{p-1,d+s-1-p}_{\mbox{\rm\footnotesize prim}}(Y)\simeq H^{p-s,d-p}_{\mbox{\rm\footnotesize prim}}(X). $$
\end{prop}

\begin{proof} One constructs $\Pj^{d+s-1}_{X,\Sigma}$ via the so-called ``Cayley trick". Let $L_1,\dots, L_s$ be the line bundles associated to the quasi-smooth hypersurfaces $X_1,\dots X_s$, and  so let $\Pj(E)$ be the projective bundle  of  $E=L_1\oplus \dots \oplus L_s$. It turns out that $\Pj(E)$ is a $d+s-1$- dimensional projective simplicial toric variety whose Cox ring  is 
$$\C[x_1,\dots,x_n,y_1,\dots y_s] $$  where $S=\C[x_1,\dots, x_n]$ is the Cox ring of $\Pj^{d}_{\Sigma}$. The  hypersurface $Y$ is cut off by the polynomial $F=y_1f_1+\dots+ y_sf_s$ and is quasi-smooth by Lemma 2.2 in \cite{Mavlyutov}. Moreover, combining Theorem 10.13 in \cite{bc} and Theorem 3.6 in \cite{Mavlyutov}, we have that 
$${H^{p-1,d+s-1-p}_{\mbox{\rm\footnotesize prim}}(Y)}\simeq { R(F)_{(d+s-p)\beta-\beta_0}} \simeq  H^{p-s,d-p}_{\mbox{\rm\footnotesize prim}}(X)$$
for  $p\neq \frac{d+s-1}{2}, \frac{d+s-3}{2} $
 as desired. \end{proof}

 Here $R(F)$ is the Jacobian ring of $Y$, i.e., the quotient of the Cox ring $$R(F) = \C[x_1,\dots,x_n,y_1,\dots y_s]/J(F),$$ where $J(F)$ is    the ideal generated by the derivatives of $F$, see \cite{bc}.


\begin{rmk} With the same notation of Proposition \ref{prop}, note that we have a well defined map
$$ 
\begin{array}{rcl}
\phi: |\beta_1|\times \dots \times |\beta_s| &\rightarrow&|\beta |\\
(f_1,\dots,f_s)&\mapsto & f_1y_1+\dots+f_sy_s.
\end{array}
$$
Moreover, by the Noether-Lefschetz theorem $  NL_{\beta}$ is a countable union of closed sets $\bigcup_i {C_i}$ and hence $\bigcup \phi^{-1}(C_i)$ is too.  
\end{rmk}
We have a Noether-Lefschetz type theorem, namely,

\begin{thm}\label{NL} Let $\Pj^d_{\Sigma}$ be an Oda projective simplicial toric variety. Then for a very general quasi-smooth intersection subvariety $X$ cut off by $f_1,\dots f_s$ such that $d+s=2(k+1)$ and $\sum_{i=1}^s \deg (f_i)-\beta_0$ is nef, one has that
$$H^{k+1-s,k+1-s}(X,\Q)=i^*\left( H^{k+1-s,k+1-s}(\Pj^d_{\Sigma},\Q)\right)$$ 

\end{thm}
So we get a natural generalization of the Noether-Lefschetz loci. 

\begin{dfn}\label{newNL} The Noether-Lefschetz locus $NL_{\beta_1,\dots,\beta_s}$ of   quasi-smooth intersection varieties is the locus of $s-$tuples $(f_1,\dots,f_s)$ such that $X=X_{f_1}\cap \dots X_{f_s}$ is quasi-smooth intersection with $f_i\in|\beta_i|$  and   $H^{k+1-s,k+1-s}(X,\Q)\neq i^*\left( H^{k+1-s,k+1-s}(\Pj^d_{\Sigma},\Q)\right)$.
\end{dfn}

Now we  consider the  Hodge conjecture  for very general  quasi-smooth intersection subvarieties in $\Pj^d_{\Sigma}$.

\begin{cor}\label{hodge1} If  $\Pj_{\Sigma}^d$ is a Oda projective simplicial toric variety,
the Hodge Conjecture holds for a very general  quasi-smooth intersection subvariety $X$ cut off by $f_1,\dots f_s$ such that $d+s=2(k+1)$ and  $\sum_{i=1}^s \deg (f_i)- \beta_0$ is nef.\end{cor} 

\begin{proof} First note that by Thereom 4.1 in \cite{Ikeda} the projective simplicial toric variety $\Pj_{X,\Sigma}^{2k+1}$ is Oda and since $X$ is very general the quasi-smooth hypersurface $Y$ is very general as well. So applying the Noether-Lefschetz theorem  one has that $h^{k,k}_{\mbox{\rm\footnotesize prim}}(Y)=0= h^{k+1-s,k+1-s}_{\mbox{\rm\footnotesize prim}}(X)$ or equivalently every $(k+1-s,k+1-s)$ cohomology class is a linear combination of algebraic cycles. 
\end{proof}

\section{Cox-Gorenstein ideals} 
We shall need a partial generalization of Macaulay's theorem (see e.g.~Thm.~6.19 in \cite{v2} for the classical theorem). This generalization is basically contained
in the work of Cox and Cattani-Cox-Dickenstein \cite{cox-res,ccd}.

 Let $S$ be the Cox ring of a complete simplicial toric variety $\Pj_\Sigma$. This is graded over the effective classes in the class group $\operatorname{Cl}(\Pj_\Sigma)$ and \cite{CoxHom}
 $$ S^\alpha \simeq H^0(\Pj_\Sigma,\cO_{\Pj_\Sigma}(\alpha)).$$
 As $\cO_{\Pj_\Sigma}(\alpha)$ is coherent and $\Pj_\Sigma$ is complete,  each $S^\alpha$ is finite-dimensional over $\C$; in particular, $S^0\simeq \C$. 

  \begin{lma} For every effective $N\in\operatorname{Cl}(\Pj_\Sigma)$, the set of classes $\alpha\in\operatorname{Cl}(\Pj_\Sigma)$ such that $N-\alpha$ is effective
 is finite. \label{finite}
 \end{lma}
 
 \begin{proof} Since the torsion submodule of $\operatorname{Cl}(\Pj_\Sigma)$ is finite, we may assume  that $\operatorname{Cl}(\Pj_\Sigma)$ is free.
 Then the exact sequence
 $$ 0 \to  M \to \operatorname{Div}_{\mathbb T}(\Pj_\Sigma) \to \operatorname{Cl}(\Pj_\Sigma) \to 0 $$
 splits, and we may identify $\operatorname{Cl}(\Pj_\Sigma)$ with a free subgroup of $ \operatorname{Div}_{\mathbb T}(\Pj_\Sigma)$, generated by a subset
 $\{D_1,\dots,D_r\}$ of $\mathbb T$-invariant divisors. A class in $\operatorname{Cl}(\Pj_\Sigma)$ is effective if and only its coefficients on this basis
 are nonnegative, whence the claim follows.
 \end{proof}

 We shall give a definition of {\em Cox-Gorenstein ideal} of the Cox rings which generalizes to toric varieties the definition
 given by Otwinowska in \cite{Annia} for projective spaces. 
Let  $B\subset S$ be the  irrelevant ideal, and for a graded ideal $I\subset B$,
denote by $V_{\mathbb T}(I)$ the corresponding closed subscheme of $\Pj_\Sigma$.
 
 \begin{dfn} \label{lambdator} 
 A graded ideal $I$ of $S$ contained in $B$ is said to be
 a Cox-Gorentstein ideal of socle degree  $N\in\operatorname{Cl}(\Pj_\Sigma)$  if
 \begin{enumerate}\itemsep=-2pt
\item there exists a $\C$-linear form
$\Lambda\in (S^N)^\vee$ such that for all $\alpha \in\operatorname{Cl}(\Pj_\Sigma)$
\begin{equation}\label{eqlambda} I^\alpha  =\{f\in S^\alpha \,\vert\, \Lambda(fg) = 0 \ \mbox{for all} \ g\in S^{N-\alpha }\}; \end{equation}
\item $V_{\mathbb T}(I)=\emptyset$.
\end{enumerate} 
\end{dfn}

\begin{rmk}  Cox-Gorenstein ideals need not be Artinian. Property 2 in this definition replaces that condition. 
\end{rmk}

\begin{prop}
Let $R=S/I$. If $I$ is Cox-Gorenstein then  \begin{enumerate}\itemsep=-2pt
\item $\dim_\C R^N = 1$; \item
the natural bilinear morphism \begin{equation}\label{pair}R^\alpha \times R^{N-\alpha } \to R^N\simeq \C \end{equation} is nondegenerate
whenever $\alpha$ and $N-\alpha$ are effective.   \end{enumerate} 
\end{prop}

\begin{proof}  
1. From eq.~\eqref{eqlambda} we see that the sequence
$$ 0 \to I^N \to A^N \xrightarrow{\Lambda} \C \to 0 $$
is exact.

2. Define
$\Phi\colon R^\alpha\times R^{N-\alpha} \to \C$ as $\Phi(x,y) = \Lambda(\bar x\bar y)$, where $\bar x$, $\bar y$ are pre-images of
$x$, $y$ in $S$. One easily checks that this is well defined and that via the isomorphism $R^N\simeq k$ it coincides
with the pairing \eqref{pair}.  Now if $x\in R^\alpha$  and $\Phi(x,y) = 0$ for all $y\in R^{N-\alpha}$
then $\Lambda(\bar x\bar y) =0$ for all $\bar y\in S^{N-\alpha}$ so that $\bar x \in I^\alpha$, i.e., $x=0$. 
  \end{proof}

Let $f_0,\dots,f_d$ be homogeneous polynomials, $f_i\in S^{\alpha_i}$, where $d=\dim \Pj_\Sigma$ and each $\alpha_i$ is ample, and let $ N = \sum_i\alpha_i-\beta_0$, where
$\beta_0$ is the anticanonical class of $\Pj_\Sigma$. Assume that the $f_i$ have no common zeroes in $\Pj_\Sigma$, i.e., $V_{\mathbb T}(I)=\emptyset$ if
$I=(f_0,\dots,f_d)$. 

In \cite{bc,cox-res,ccd}  it is shown that  for each $G \in S^N$ one can define a meromorphic 
$d$-form $\xi_G$ on $\Pj_\Sigma$ by letting
$$ \xi_G = \frac{G\,\Omega}{f_0\cdots f_d}$$
where $\Omega$ is a Euler form on $\Pj_\Sigma$. The form $\xi_G$ determines a class in $H^d(\Pj_\Sigma,\omega)$,
where $\omega$ is the canonical sheaf of $\Pj_\Sigma$ (the sheaf of Zariski $d$-forms on $\Pj_\Sigma$), and in turn
the trace morphism  $\operatorname{Tr}_{\Pj_\Sigma}\colon H^d(\Pj_\Sigma,\omega)\to\C$ associates a complex number to $G$, so we can
define $\Lambda\in (S^N)^\vee$ as 
\begin{equation}\label{deflambdator} \Lambda(G) = \operatorname{Tr}_{\Pj_\Sigma}([\xi_G])\in \C.\end{equation}

Finally, we can prove a toric version of Macaulay's theorem.

\begin{thm} 
The linear map defined in Eq.~\eqref{deflambdator} 
satisfies the condition in Definition \ref{lambdator}. Therefore, the ideal $I=(f_0,\dots,f_d)$ is a Cox-Gorenstein ideal of socle degree $N$.
\label{toricmac}
\end{thm}
\begin{proof} By Prop.~3.13 in \cite{ccd} the map $\Lambda$ establishes an isomorphism $R^N \simeq \C$. Hence,
if $f\in S^\alpha$ is such that $\Lambda(fg)=0$ for all $g\in S^{N-\alpha}$, then $fg\in I^N$, which implies $f\in I^\alpha$.
On the other hand, it is clear that  $\Lambda(fg)=0$ if $f\in I^\alpha$ and $g\in S^{N-\alpha}$.
\end{proof} 

Another example is given in terms of {\em toric Jacobian ideals.}
For every ray $\rho\in\Sigma(1)$ we shall denote by $v_\rho$ its rational generator,
and by $x_\rho$ the corresponding variable in the Cox ring. Recall that $d$ is the dimension of the toric variety $\Pj_\Sigma$, while we denote
by $r=\#\Sigma(1)$ the number of rays. 
Given $f\in S^\beta$ one defines its {\em toric Jacobian ideal}   as 
$$J_0(f) = \left( x_{\rho_1} \frac{\partial f }{\partial  x_{\rho_1}}, \dots,  x_{\rho_r} \frac{\partial f }{\partial  x_{\rho_r}} \right).$$

We recall from
\cite{bc} the   definition of nondegenerate hypersurface and some properties (Def.~4.13 and Prop.~4.15).

\begin{dfn} Let $f\in S(\Sigma)^\beta$, with $\beta$ an ample Cartier class. The associated hypersurface $X_f$ is nondegenerate if for all $\sigma\in\Sigma$ the affine hypersurface $X_f\cap O(\sigma)$ is a smooth codimension one subvariety of the orbit $O(\sigma)$ of the action of the torus
$\mathbb T^d$. 
\end{dfn}

\begin{prop} \begin{enumerate} \itemsep=-3pt\item Every nondegenerate hypersurface is quasi-smooth.
\item If $f$ is generic then $X_f$ is nondegenerate.
\end{enumerate} 
\end{prop}

The  following is part of Prop.~5.3 in \cite{cox-res}, with some changes in the terminology.

\begin{prop} Let $f\in S(\Sigma)^\beta$, and let $\{\rho_1,\dots,\rho_d\}\subset\Sigma(1)$ be such that
$v_{\rho_1},\dots,v_{\rho_d}$ are linearly independent.  \begin{enumerate}  \item The toric Jacobian ideal of $f$ coincides with the ideal 
$$ \left( f, x_{\rho_1} \frac{\partial f }{\partial  x_{\rho_1}}, \dots,  x_{\rho_d} \frac{\partial f }{\partial  x_{\rho_d}} \right).$$
\item The following conditions are equivalent:
\begin{enumerate} \item $f$ is nondegenerate; \item the polynomials $x_{\rho_i} \frac{\partial f }{\partial  x_{\rho_i}}$, $i=1,\dots,r$, do not vanish simultaneously on $X_f$; \item
the polynomials  $f$ and $x_{\rho_i} \frac{\partial f }{\partial  x_{\rho_i}}$, $i=1,\dots,d$, do not vanish simultaneously on $X_f$.
\end{enumerate}
\item If moreover $\beta$ is ample and $f$ is nondegenerate, then $J_0(f)$ is a   Cox-Gorenstein ideal of socle degree
$N = (d+1)\beta-\beta_0$, where $\beta_0$ is the anticanonical class of $\Pj_\Sigma^d$.
\end{enumerate}
\end{prop}

\section{Asymptotic Hodge conjecture}

Let us recall  part of the notation and assumptions  of \cite{bm}. Let $\Pj^{2k+1}_\Sigma$ be an Oda variety with an ample Cartier class $\beta$ such that $k\beta-\beta_0=n\eta$, where
$\beta_0$ is the anticanonical class, $\eta$ is a primitive ample class and $n\in\N$.  
Let $X_f\in\vert\beta\vert$  be  a quasi-smooth hypersurface in the Noether-Lefschetz locus associated to a nontrival Hodge class $\lambda\in F^k\mathcal H^{2k}(U)$.
Let $r$ be number of rays of $\Sigma$, so that $r\ge 2(k+1)$. Assuming that $n$ is big enough, it follows from Proposition 4.7 or Theorem 6.1 in \cite{bm}
that there exists a $k$-dimensional subvariety $V$ of $X_f$  satisfying the following conditions:
\begin{itemize}
	 	\item $\deg V\leq \min\{ 2\delta m_\beta,d\}$  with $0< \delta< \frac{1}{4(r-(k+1))}$ (the number $m_\beta$ was defined in Eq.~\eqref{mbeta});
	\item the graded ideals $I_V$ and 
	\begin{equation}\label{E} E=\{g\in S^{\bullet} \mid  \sum_{i=1}^b \lambda_i \int_{\operatorname{Tub}\gamma_i}  \frac{gh\Omega_0}{f^{k+1}}=0 
	\ \mbox{for all}\ h\in S^{N-\bullet} \} ,\end{equation}
 coincide in degree less than or equal to $\left( m_{\beta}-2-(r-j)\deg V\right)\eta $ for some $j$, with  $0<j<r$.
 Here $\operatorname{Tub}(-)$ is the adjoint of the residue map, and $N=(k+1)\beta-\beta_0$ is the socle degre of the  {\em Cox-Gorenstein ideal} $E$, while
 $$\lambda_f =\left(\sum_{i=1}^b\,\lambda_i\gamma_i\right)^{pd}$$ is the Poincar\'e dual of some rational combination of the homology cycles $\gamma_i$ generating
 $H_{2k}(X_f,\Q)$. Moreover, via the isomorphism $T_fU\simeq S^{\beta}$, the degree $\beta$ summand $E^\beta$ of $E$ is identified with the tangent space
 $T_fN^{k,\beta}_{\lambda,U}$ to the Noether-Lefschetz locus, so that $E^\beta$ contains  the degree $\beta$ part $J(f)^\beta$ the Jacobian ideal of $f$.
\end{itemize}

\begin{lma} The toric Jacobian ideal $J_0(f)$ is contained in $E$.
\end{lma}
\begin{proof} $J_0(f)\subset J(f)$, so that $J_0(f)^\beta\subset J(f)^\beta\subset E^\beta$, and since $J_0(f)$ is generated in degree $\beta$, one has
$J_0(f)\subset E$. 
\end{proof}

We denote by  $\lambda_V$ the class of $V$ in $H^{k,k}_{\mbox{\rm\footnotesize prim}}(X_f,\Q)$.
 
\begin{thm}\label{hodge2} If $V$ is a smooth   intersection subvariety,   there exists $c\in \Q^*$ such that $\lambda_f=c\lambda_V$.

\end{thm}

\begin{proof} We divide the proof in three steps. 

\textbf{Step I}: $\lambda_V \neq 0$.  
Since $V\subset X_f$ is a regular embedding we have   
$$
\begin{array}{ccl}
[V]^2_{X_f}&=&\int_{V} c_k(N_{V/X_f})\\
&=&\int_{V} c_k(N_{V/\Pj_{\Sigma}})/{c_k({N_{X_f/\Pj_{\Sigma}}}_{|V})}\\
&=& \deg  V \left( \mbox{coefficient of }\; t^k \ {\rm in}\; \frac {\prod_i (1+\deg (A_i)t)}{1+\deg (X_f)t} \right)
\end{array}
$$
By contradiction if $\deg (X_f)[V]=\deg  V \cdot  c_1^k (\cO_{X_f}(\eta))$ then $$\deg (X_f)^2[V]^2_{X_f}=(\deg V)^2 c_1^{2k}(\cO_{X_f}(\eta))=(\deg V)^2 \deg(X_f)$$ which implies that $\beta =\deg(X_f)$ is a proper divisor of $\deg V$, which is a contradiction, so that    Step I is proved.

\textbf{Step II}. Let $E_{\rm alg}$ and $E$ be 
 the Cox-Gorenstein ideal  associated to $\lambda_V$ and  $\lambda_f$, respectively, as in equation \eqref{E}.

 To prove the theorem  it is enough to show that $E=E_{\rm alg}$.  
Note that $I_V+J_0(f)$ is contained in $E$ and $E_{\rm alg}$. Moreover,  since $V\subset X_f$, and $f$ is quasi-smooth,
there exist $K_1,\dots K_{k+1}\in B$ such that 
$f=A_1K_1+\dots A_{k+1}K_{k+1}$ and 
$(A_1,\dots ,A_{k+1}, K_1,\dots K_{k+1})$ is a   Cox-Gorenstein ideal with socle degree $N$ ; this will follow from the next step,  which concludes the proof.

\textbf{Step III}. 
It is enough to show that every Cox-Gorenstein ideal $\mathcal{I}$ of socle degree $N$ containing $I_V+J_0(f)$ also contains $(A_1,\dots ,A_{k+1}, K_1,\dots K_{k+1})$. By assumption
$$\left(A_j,j\in\{1,\dots, k+1\}, \sum_{j=1}^{k+1}x_i\frac{\partial A_j}{\partial x_i}K_j, i\in {1,\dots, r} \right) \subset \mathcal{I}.$$ 
Let us see that $K_j\in \mathcal{I}$ for every $j\in\{1,\dots,k+1\}$. Let $M\in \operatorname{Mat}(r\times(k+1))$ be the matrix $[x_i\frac{\partial A_j}{\partial x_i}]$ and $K$ the column $(K_j)_{j\in\{1,\dots,k+1\}}$. Let $I\subset \{1,\dots r\}$ with cardinality $k+1$ and let $M_I$ be the matrix  obtained extracting the $i\in I$-arrows of $M$. We have that $\sum_{j=1}^{k+1}x_i\frac{\partial A_j}{\partial x_i}K_j=(MK)_i=(M_IK)_i$; multiplying by the adjoint of $M_I$ we get that $\det (M_I)K_j\in \mathcal{I}$   for all $j\in\{1,\dots k+1\}$. On one hand the ideal $(\mathcal{I}:K_j)$ contains the ideal 
$$\mathcal{J}=I_V+\left< \det M_I\,\vert\, I\subset \{1,\dots,r\},\ \# I = k+1\right>.$$
Since $V$ is a smooth complete intersection subvariety, it follows that $\mathcal J$ is base point free, and therefore it contains a complete intersection Cox-Gorenstein ideal $\mathcal J'$ 
by the toric Macaulay theorem, Theorem \ref{toricmac}. Since $\mathcal J$ is generated in degree less than or equal to $\deg V \cdot\eta^k$, we can take $\mathcal J'$ with the same property. It follows that
$$ soc(\mathcal J') \le 2(k+1)(\deg V)\eta - \beta_0 \le 2rm_\beta\delta\eta-\beta_0.$$

On the other hand if $K_j\nin \mathcal{I}$ then $(\mathcal{I}: K_j)$ contains a Cox-Gorenstein ideal with socle degree 
$$N-\deg K_j\geq N-\beta= k\beta-\beta_0; $$ then comparing the above two inequalities and keeping in mind that $r\geq 2(k+1)$, we get   $$\delta>\frac{1}{2r}>\frac{1}{4(r-(k+1))},$$ which is absurd.
\end{proof}

\end{document}